\documentclass{ourlematema}
\usepackage{graphicx} \usepackage{amsmath,amsfonts,amssymb,amsthm}
\usepackage{hyperref}
\hypersetup{
pdftitle={Cyclic polytopes through the lens of iterated integrals},
pdfsubject={The volume of a cyclic polytope can be obtained by forming an iterated integral along a suitable piecewise linear path running through its edges. Different choices of such a path are related by the action of a subgroup of the combinatorial automorphisms of the polytope. Motivated by this observation, we look for other linear combinations of iterated integrals that are invariant under the subgroup action. This yields interesting polynomial attributes of the cyclic polytope. We prove that there are infinitely many of these invariants which are algebraically independent in the shuffle algebra.},
pdfauthor={Felix Lotter, Rosa Preiß},
pdfkeywords={piecewise linear paths; shuffle algebra; permutation group action; signed volume; invariants; positive matrices}
}
\usepackage{tikz-cd}
\usepackage{cleveref}
\usepackage{enumerate}
\usepackage{xcolor}
\usepackage{shuffle}

\newtheorem{theorem}{Theorem}[section]
\newtheorem*{theorem*}{Theorem}
\newtheorem{corollary}[theorem]{Corollary}
\newtheorem{remark}[theorem]{Remark}
\newtheorem{lmm}[theorem]{Lemma}
\newtheorem{conjecture}[theorem]{Conjecture}

\theoremstyle{definition}
\newtheorem{definition}[theorem]{Definition}
\newtheorem{definitionthm}[theorem]{Theorem and Definition}

\newtheorem{proposition}[theorem]{Proposition}
\newtheorem{example}[theorem]{Example}
\newtheorem{myproblem}[theorem]{Problem}
\Crefname{thm}{Theorem}{Theorems}
\newcommand{\word}[1]{\texttt{#1}}
\definecolor{darkgreen}{RGB}{0,150,20}

\newcommand{\annotationR}[1]{\textcolor{purple}{#1}}
\newcommand{\mb}{\mathbb}
\newcommand{\mc}{\mathcal}
\newcommand{\R}{{\mb R}}
\newcommand{\convperms}{C}
\newcommand{\spann}{\operatorname{span}}
\newcommand{\ourfeatures}{\mathsf{Inv}}
\newcommand{\pwlinear}{\mathsf{PL}}
\newcommand{\signedvolume}{\mathsf{vol}}
\newcommand{\antipode}{\mathcal{A}}
\newcommand{\timerevinv}{\mathsf{TimeRevInv}}
\newcommand{\loopclosureinv}{\mathsf{LoopClosureInv}}
\newcommand{\convhull}{\operatorname{conv}}
\newcommand{\homomor}{H}
\newcommand{\letteri}{\texttt{i}}
\newcommand{\emptyword}{\texttt{e}}

\DeclareMathOperator{\Aut}{Aut}
\DeclareMathOperator{\sgn}{sgn}

\title{Cyclic polytopes through the lens of iterated integrals}
\author{Felix Lotter}
\address{Max Planck Institute for Mathematics in the Sciences, Leipzig\\ e-mail: \texttt{felix.lotter@mis.mpg.de}}
\author{Rosa Preiss}
\address{Technical University Berlin\\ e-mail: \texttt{preiss@math.tu-berlin.de}}
\date{10/15/24}

\begin{document}
\maketitle
\begin{abstract}
The volume of a cyclic polytope can be obtained by forming an iterated integral along a suitable piecewise linear path running through its edges. Different choices of such a path are related by the action of a subgroup of the combinatorial automorphisms of the polytope. Motivated by this observation, we look for other linear combinations of iterated integrals that are invariant under the subgroup action. This yields interesting polynomial attributes of the cyclic polytope. We prove that there are infinitely many of these invariants which are algebraically independent in the shuffle algebra.
\\[2ex]
\emph{Journal reference}: Special volume on Positive Geometry, Le Matematiche 80 (1) (2025), 365-385\\\emph{AMS 2010 Subject Classification:} 60L10, 13A50, 52B05 \\\emph{Keywords:} piecewise linear paths, shuffle algebra, permutation group \mbox{action}, signed volume, invariants, positive matrices
\end{abstract}

\section{Introduction}
\paragraph{Iterated integrals and piecewise linear paths}
A \emph{path}, for the purpose of this paper, is a continuous map $X:[0,1]\to \mb R^d$ such that the coordinate functions $X_i$ are piecewise continuously differentiable. Given such a path $X$, its \emph{(iterated integral) signature} is the linear form
\begin{align}\label{def:sig}
    S(X): \mb R\langle \texttt{1},\dots, \texttt{d}\rangle  \to \mb R, \quad\letteri_1\cdots \letteri_k &\mapsto \int_{\Delta_k} dX_{i_1}(t_1)\dots dX_{i_k}(t_k)
\end{align}
where $k$ varies over all positive integers and $\Delta_k$ denotes the simplex $0\leq t_1 \leq \dots \leq t_k \leq 1$. Here, $\mb R\langle \texttt{1},\dots, \texttt{d}\rangle$ is the free associative algebra over the letters (that is, formal symbols) $\mathtt{1},\dots,\mathtt{d}$. The words $\letteri_1\dots\letteri_k$ form a basis of this space, such that \eqref{def:sig} does indeed define a linear form. For example, $S(X)(\texttt{1} \texttt{1} + \texttt{1} \texttt{2}) = \int_0^1 \int_0^{t_2} X_1'(t_1) X_1'(t_2) + X_1'(t_1) X_2'(t_2) dt_1 dt_2$.

The signature of a path determines the path up to translation, reparametrization and tree-like equivalence \cite{BGLY16,Chen1958}.\par
In this paper, we are interested in \emph{piecewise linear paths}. Such a path is uniquely determined by its control points $x_1,\dots,x_n \in \mb R^d$, that is, the (ordered) set of start and end points of all of its linear segments.
The signature of such a piecewise linear path can be described explicitly in terms of the increments $a_i := x_i - x_{i-1}$.
\pagebreak

\noindent In fact, it defines a map
\begin{equation}\label{eq:sign map}
   \homomor^d_n: \mb R\langle \texttt 1,\dots,\texttt d \rangle \to \mb R[x_1,\dots,x_n],
\end{equation}
where $x_i=(x_{i1},\dots,x_{id})$,
see \Cref{def:Hnd_via_sig} and \Cref{eq:Hnd_recursive1} for a recursive formula.

If the left-hand side is viewed as a commutative algebra $\mb R\langle \texttt 1,\dots,\texttt d \rangle_\shuffle$ via the \textit{shuffle product} $\shuffle$ (see \eqref{def:shuffle}, and e.g.\ \cite{colmenarejopreiss20} for an introduction based on the recursive definition), this map becomes a homomorphism of graded algebras. Its image is a subalgebra of $\mb R[x_1,\dots,x_n]$ which we will call \emph{the ring of signature polynomials in $d \times n$ variables} in the following, denoted by $\mc S^d[x_1,\dots,x_n]$.\par
The polynomials in $\mc S^d[x_1,\dots,x_n]$ inherit some nice properties from their integral representation. For example, they are translation invariant,
\begin{equation*}
p(x_1,\dots,x_n)=p(x_1+y,\dots,x_n+y)
\end{equation*}
for $y \in \mb R^d$. Moreover, if $p \in \mc S^d[x_1,\dots,x_n]$,
then for $1 < i < n$ the polynomial
\begin{equation}\label{eq:face map}
p(x_1,\dots,x_{i-1},\lambda x_{i-1} + (1-\lambda) x_{i+1},x_{i+1},\dots,x_n)
\end{equation}
is independent of $\lambda \in [0,1]$, due to reparame\-trization invariance of iterated integrals. In particular, \eqref{eq:face map} is a polynomial in variables $x_j, j\not=i$.
\par
It follows that for an injective map $I \to J$ of finite totally ordered sets, there is a natural ``restriction map" $S^d[x_J] \to \mc S^d[x_I]$ (where $x_I = \{x_i \ | \ i \in I \}$), given by the map $\mc S^d[x_J] \to \mc S^d[x_I]$ induced by replacing for $j \in J \backslash I$ the variable $x_j$ by $x_{i}$ where $i$ is the largest element in $I$ smaller than $j$ or, by the above equivalently, the smallest element in $I$ larger than $j$. \par
Now, for fixed $n$ and a given subgroup $G$ of $S_n$ one can ask the following question: Which signature polynomials in $d\times n$ variables are invariant under the action of $G$ on $x_1,\dots,x_n$ by permutation? More precisely, we would like to determine the pullback $\mathsf{Inv}^d_n(G) \subseteq \mb R\langle \texttt 1,\dots,\texttt d \rangle$ in
\[\begin{tikzcd}
    \mathsf{Inv}^d_n(G) \rar \dar & \mb R[x_1,\dots,x_n]^G \dar \\
    \mb R\langle \texttt 1,\dots,\texttt d \rangle \rar["H^d_n"] & \mb R[x_1,\dots,x_n]
\end{tikzcd}\]
where $\mb R[x_1,\dots,x_n]^G$ is the subring of $G$-invariants in $\mb R[x_1,\dots,x_n]$.

\paragraph{Towards positivity}

In this paper, we address this question for a specific choice of $G$. Namely, given $d$ and $n\geq d+1$, $S_n$ acts naturally by permutations of columns on the set of $(d+1) \times n$-matrices
\begin{equation}\label{eq:matr}
  \begin{pmatrix}
  1 & \dots & 1\\
   x_1 & \dots & x_n
  \end{pmatrix}
 \end{equation}
 and we want to choose $G$ as the stabiliser $\convperms^d_n$ of the subset of matrices whose maximal minors are positive. Following the terminology of \cite[Section 2]{arkani2014amplituhedron}, we call these matrices \textit{positive}.\par
 Our motivation is that for each such positive matrix, the volume of the polytope $\convhull(x_1,\dots,x_n)$ can be obtained from the signature of the piecewise linear path $X$ with control points $x_1 \to \dots \to x_n$ as the \textit{signed volume}
 $$\langle S(X), \tfrac{1}{d!}\,\mathsf{vol}_d \rangle = \frac{1}{d!}\int_{\Delta_d} \det \begin{pmatrix}
     X'(t_1) & \dots & X'(t_d)
 \end{pmatrix} dt_1 \ldots dt_d$$ where
 \begin{equation}\label{eq:signed vol}
    \signedvolume_d := \sum_{\sigma \in S_d} \sgn(\sigma) \sigma(\texttt 1)\dots \sigma(\texttt d) \in \mb R\langle \texttt{1},\dots, \texttt{d}\rangle,
\end{equation}
see \cite[Section~3.3]{diehl2019invariants}. In fact, one can deduce more generally that the volume of the convex hull of a so-called \textit{convex} path agrees with its signed volume, cf. \cite[Theorem 3.4]{amendolaleemeroni23}. Convex piecewise linear paths correspond precisely to matrices with nonnegative maximal minors \cite[Section~2]{ANudelman_1975}. See the references in \cite{amendolaleemeroni23} for more context on convex paths. 

As the set of $x_1, \dots, x_n$ with \eqref{eq:matr} positive is Zariski-dense in the set of all $x_1, \dots, x_n$, it follows that $\signedvolume_d \in \mathsf{Inv}^d_n(\convperms^d_n)$ for all $n\geq d+1$ (see \Cref{prop:eq rel on pwl}).  Thus, we expect $\mathsf{Inv}^d_n(\convperms^d_n)$ to describe geometric features of polytopes of the form $\convhull(x_1,\dots,x_n)$ for positive matrices \eqref{eq:matr}. 
Such polytopes are known as cyclic $d$-polytopes admitting a \textit{canonical labeling} $x_1,\dots,x_n$ (note however that this is not necessarily unique).
\begin{figure}[h]
    \centering
    \includegraphics[width=0.8\linewidth]{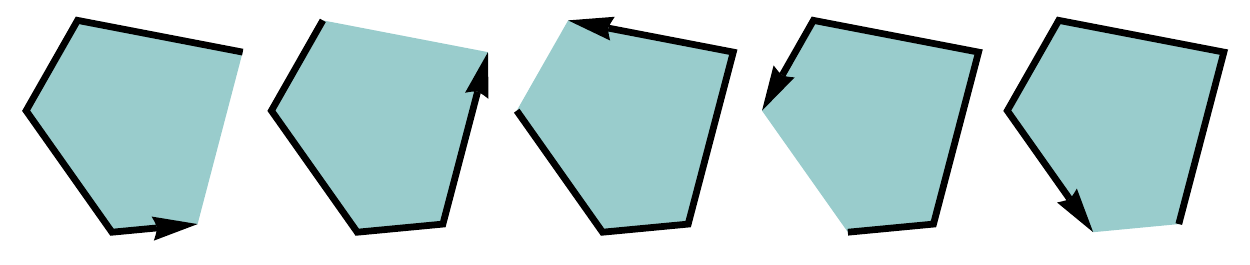}
    \caption{A cyclic $2$-polytope with $5$ vertices, spanned by $5$ different piecewise linear paths, related by cyclic permutations of their control points. The volume of the polygon agrees with the signed volume of each of the paths.}\label{fig:2polytope}
\end{figure}
\par
Of particular interest is the intersection
$$\mathsf{Inv}^d := \bigcap_{n\geq d+1} \mathsf{Inv}^d_n(\convperms^d_n)$$
which we call the \textit{ring of volume invariants}. This terminology will be motivated in \Cref{prop:eq rel on pwl} where we show that if the signed volume of a piecewise linear path in general position is invariant under a permutation of the control points then its signature at any $w \in \ourfeatures^d$ will be as well.\par
Note that $\ourfeatures^d$ forms a subalgebra of $\mb R\langle \texttt{1},\dots, \texttt{d}\rangle_\shuffle$. One might view $H_n^d(\ourfeatures^d)$ as functions on the set of canonically labeled cyclic $d$-polytopes with $n$ vertices. In contrast to the discussion in \cite[Section~6]{fedorchukpak}, these features are not necessarily SO-invariants. Note that by definition of $\ourfeatures^d$, $\homomor^d_{\bullet}(\ourfeatures^d)$ is compatible with the restriction maps of rings of signature polynomials. We can interpret them as restrictions to subpolytopes.\par
As noted above, we certainly have $\signedvolume_d \in \mathsf{Inv}^d$. The main result of this paper is the following theorem in \Cref{sec:ring_of_invariants}, showing that there is an abundance of volume invariants:

\begin{theorem*}[{\ref{thm:abundant invariants}}]
    For any $d$, $\ourfeatures^d$ contains infinitely many algebraically independent elements (with respect to the shuffle product) and is thus in particular infinitely generated as a (shuffle) subalgebra of $\R\langle\word{1},\dots,\word{d}\rangle$.
\end{theorem*}

\paragraph{Outline}
In Section \ref{sec:Hnd}, we precisely define and thoroughly discuss the ring homomorphism $H_n^d$.
In particular, we describe how to obtain a useful recursive formula in \eqref{eq:Hnd_recursive1}.

Section \ref{sec:Cnd} is devoted to introducing the subgroup $C_n^d\subset S_n$ as the stabilizer of positive matrices, i.e.\ matrices with positive maximal minors, under column permutation. We then show that $C_n^d$ is exactly the subgroup of $S_n$ under which the signed volume is invariant for piecewise linear paths in general position.

Finally, in Section \ref{sec:ring_of_invariants}, we prove our main results. 
In Propositions \ref{prop:inv geq3 odd} and \ref{prop:inv_loopclosure_timerev}, we fully characterize the invariant rings $\ourfeatures_{\geq d+3}^d$ for $d+3$ and more points in $\mb R^d$. 
In Theorem \ref{thm:abundant invariants}, we show that the rings of volume invariants $\ourfeatures^d$ are `very large', 
in the sense that they are infinitely generated, and even contain infinitely many algebraically independent elements.

\section{Piecewise linear paths and signatures}\label{sec:Hnd}

Our goal in this section is to define and explain the homomorphism $H_n^d$.\par
\begin{definition}
    A piecewise linear path with control points $x_1,\dots,x_n \in \mb R^d$ is a continuous map $X: [0,1] \to \mb R^d$ such that there are $0=t_1\leq t_2 \leq \dots \leq t_n=1$ with the property that $X(t_i)=x_i$ for all $i$ and $X$ is an affine map on all intervals $[t_i,t_{i+1}]$. 
    We write $\{x_1\to \dots \to x_n\}$ to denote such a path independent of the precise time parametrization,
    and we write $\pwlinear^d_n$ for the set of all piecewise linear paths through $\mb R^d$ with $n$ control points.
\end{definition}

In particular, a piecewise linear path with two control points is a linear path. Up to reparametrization, any piecewise linear path can be viewed as a \textit{concatenation} of linear paths.
\begin{definition}
    Given paths $X:[0,1]\to \mb R^d$ and $Y:[0,1] \to \mb R^d$, their concatenation is the path $X \sqcup Y:[0,1] \to \mb R^d$ which is defined as $X(2t)$ for $t\in[0,\frac 1 2]$ and as $Y(2t - 1)$ for $t\in [\frac 1 2, 1]$.
\end{definition}

An important property of the signature is its compatibility with concatenation, in the following sense:

\begin{proposition}[Chen's identity, Theorem~3.1 of \cite{bib:Che1954}]\label{prop:chen}
    Let $X$ and $Y$ be paths in $\mb R^d$. Then 
    $$S(X \sqcup Y) = S(X) \bullet S(Y).$$
\end{proposition}
\noindent Here $S(X) \bullet S(Y)$ is the composition
    \[\begin{tikzcd}[nodes={inner sep=10pt}]
        \mb R\langle \mathtt{1}, \dots, \mathtt{d} \rangle \rar{\Delta} &  \mb R\langle \mathtt{1}, \dots, \mathtt{d} \rangle \otimes \mb R\langle \mathtt{1}, \dots, \mathtt{d} \rangle \rar{S(X) \otimes S(Y)}  & \mb R
    \end{tikzcd}\]
    where $\Delta$ denotes the coproduct of the Hopf algebra $\mb R\langle \mathtt{1}, \dots, \mathtt{d} \rangle$. More explicitly, $S(X)\bullet S(Y)$ maps a word $\letteri_1\cdots \letteri_k$ to
    $$\sum_{j=0}^k \langle S(X), \letteri_1\cdots \letteri_j \rangle \langle S(Y), \letteri_{j+1} \cdots \letteri_k \rangle.$$

We are now ready to define $H_n^d$.
\begin{definitionthm}\label{def:Hnd_via_sig}
    For any $w\in\R\langle \texttt{1},\dots,\texttt{d}\rangle$ the function
    \begin{align*}
        f_n(w): \R^{d \times n} &\to\R,\quad (x_1,\ldots,x_n) \mapsto \langle S(\{x_1\to\dots\to x_n\}),w\rangle
    \end{align*}
    is given by a polynomial in $x_1, \ldots, x_n$. We define
    \begin{align*}
    H_n^d: \mb R\langle\texttt{1},\dots,\texttt{d}\rangle \to \mb R[x_1,\dots,x_n],\quad w &\mapsto f_n(w)
    \end{align*}
\end{definitionthm}
\begin{proof}
    First note that $f_n(w)$ is well-defined by reparametrization invariance of the signature $S$.
    We will now proceed by induction, starting with $n=2$.\par
    The signature of a linear path $X$ with control points $x_1,x_2$ is easily calculated. Indeed, note that the integrals \eqref{def:sig} only depend on the vector $a:=x_2-x_1$. The integrand of $\langle S(X), \letteri_1\cdots \letteri_k\rangle$ is just the product $a_{i_1}\ldots a_{i_k}$ and thus 
\begin{equation*}\label{eq:H2}
    f_2({\letteri_1\cdots \letteri_k}) = \frac{1}{k!} a_{i_1}\ldots a_{i_k}
\end{equation*}
as the simplex $\Delta_k$ has volume $\frac{1}{k!}$,
proving the base case. Now by \Cref{prop:chen} we have
\begin{align*}
    &f_n(\letteri_1\cdots \letteri_k)(x_1,\dots,x_n) = \\&\sum_{j=0}^k f_{l}(\letteri_1\cdots \letteri_j)(x_1,\dots,x_{l}) \cdot f_{n-l+1}(\letteri_{j+1} \cdots \letteri_k)(x_{l},\dots,x_{n})
\end{align*}
for any $n$ and all $l$. Choosing $l=2$, we conclude by induction.
\end{proof}

Note that the proof yields the recursive formula
\begin{align}\label{eq:Hnd_recursive1}
  \begin{split}
  &H_n^d(\letteri_1\cdots \letteri_k)(x_1,\dots,x_n)=\\
  &\sum_{j=0}^k\frac{1}{j!}H_{n-1}^d(\letteri_{j+1}\cdots \letteri_k)(x_2,\dots,x_n)\prod_{m=1}^j (x_{2,i_m}-x_{1,i_m}).
  \end{split}
\end{align}
for $H_n^d(\letteri_1\cdots \letteri_k)(x_1,\dots,x_n)$.

\begin{example}
 We have
 \begin{align*}
  &H_3^3(\texttt{123})(x_1,x_2,x_3) = \\
  &H_2^3(\texttt{123})(x_2,x_3)H_2^3(\emptyword)(x_1,x_2) +H_2^3(\texttt{23})(x_2,x_3)H_2^3(\texttt{1})(x_1,x_2)\\
   &\hphantom{=}+H_2^3(\texttt{3})(x_2,x_3)H_2^3(\texttt{12})(x_1,x_2)
 +H_2^3(\emptyword)(x_2,x_3)H_2^3(\texttt{123})(x_1,x_2)\\
  &=\frac{1}{3!}a_{2,1}a_{2,2}a_{2,3}
  +\frac{1}{2!}a_{2,2}a_{2,3}\cdot a_{1,1}
  +a_{2,3}\cdot\frac{1}{2!}a_{1,1}a_{1,2}
  +\frac{1}{3!}a_{1,1}a_{1,2}a_{1,3}
 \end{align*}
where $\emptyword$ is the unit of $\R\langle\texttt{1},\dots,\texttt{d}\rangle$,
the so-called empty word,
which is mapped by all $H_n^d$ to the unit constant polynomial.
\end{example}
\annotationR{}

\begin{remark}In general, $H^d_n$ can be shown to factor as

\[\begin{tikzcd}
\mb R\langle\texttt{1},\dots,\texttt{d}\rangle \dar \drar["H_n^d"] & \\
\mathrm{Qsym}(\mb R[a_{1},\dots,a_{n-1}]) \rar & \mb R[x_{1}, \dots,x_{n}]
\end{tikzcd}\]

where $\mathrm{Qsym}(\mb R[a_{1},\dots,a_{n-1}])$ denotes the ring of \textit{quasi-symmetric functions of level $d$} in the vectors $a_{1},\dots,a_{n-1}$, cf. \cite{diehleftapia2020acta}, and the bottom map sends $a_{i}$ to $x_{i+1} - x_{i}$. We refer to \cite{AFS18} for further details about the signature of a piecewise linear path.
\end{remark}

Let us now explain how to turn $\homomor_n^d$ into an algebra homomorphism. We do this by equipping the $\mb R$-vector space  $\mb R\langle \mathtt{1}, \dots, \mathtt{d} \rangle$ with the commutative \textit{shuffle product} $\shuffle$. This can be defined on words in the following way:
\begin{equation}\label{def:shuffle}
    \letteri_1\cdots \letteri_l \shuffle \letteri_{l+1}\cdots \letteri_k := \sum_{\sigma \in G} \letteri_{\sigma^{-1}(1)} \cdots \letteri_{\sigma^{-1}(k)}
\end{equation}
where $G$ is the set of $\sigma \in S_k$ with $\sigma^{-1}(1) < \dots < \sigma^{-1}(l)$ and $\sigma^{-1}(l+1) < \dots < \sigma^{-1}(k)$. In other words, $\letteri_1\cdots \letteri_l \shuffle \letteri_{l+1}\cdots \letteri_k$ is the sum of all ways of interleaving the two words $\letteri_1\cdots \letteri_l$ and $\letteri_{l+1}\cdots \letteri_k$.

\par
The commutative algebra $(\mb R\langle \mathtt{1}, \dots, \mathtt{d} \rangle, \shuffle)$ is well-understood. As an algebra, it is (infinitely) freely generated by the \textit{Lyndon words}. For details, we refer to e.g.\ \cite{reutenauer1993free}. The connection to iterated integrals is the following:
\begin{proposition}[Ree's shuffle identity \cite{Ree58}]
    Let $X$ be a path in $\mb R^d$ and $p,q \in \mb R\langle \mathtt{1}, \dots, \mathtt{d} \rangle$. Then
    $$\langle S(X), p \shuffle q \rangle = \langle S(X), p \rangle \langle S(X), q \rangle$$
\end{proposition}
\begin{corollary}
    The maps $\homomor^d_n$ are algebra homomorphisms
    $$\mb R\langle \mathtt{1}, \dots, \mathtt{d} \rangle_\shuffle := (\mb R\langle \mathtt{1}, \dots, \mathtt{d} \rangle, \shuffle) \to (\mb R[x_1,\ldots,x_n], \cdot).$$
\end{corollary}
In particular, the kernel of $\homomor^d_n$, which we denote by  $\mathcal{I}(\pwlinear_n^d)$, is an ideal in $\mb R\langle \mathtt{1}, \dots, \mathtt{d} \rangle_\shuffle$. Geometrically, it can be viewed as the vanishing ideal of $S(\pwlinear_n^d) \subseteq \mathrm{Spec} \ \mb R\langle \mathtt{1}, \dots, \mathtt{d} \rangle$, which is the image of $\pwlinear^d_n$ under the signature (in \cite{preiss2024algebraic}, this is just called the vanishing ideal of $\pwlinear_n^d$). It follows that the ring of signature polynomials $\mc S^d[x_1,\dots,x_n]$ is isomorphic to $\mb R\langle \mathtt{1}, \dots, \mathtt{d} \rangle_\shuffle/\mathcal{I}(\pwlinear_n^d)$.

\section{The stabiliser of positive matrices}\label{sec:Cnd}

In this section, our goal is to determine the group $\convperms^d_n$ from the introduction. Recall that $\convperms^d_n$ is defined as the subgroup of $S_n$ stabilising the set of positive $d\times n$-matrices under the action on columns.
\begin{definition}[{cf. \cite[Section 2]{arkani2014amplituhedron}}]
    An $d \times n$ matrix is called positive if all its maximal minors are positive.
\end{definition}

\begin{remark}
    The \textit{totally positive Grassmanian} is the quotient of positive matrices by the left $GL^+$-action, see \cite[Definition~3.1]{postnikov2006}. However, not every positive matrix is \textit{totally positive} in the sense of loc.\ cit.
\end{remark}

The group $S_n$ acts on the columns of an $(d+1) \times n$-matrix by permutation. $\convperms^d_n$ is defined as the stabiliser of positive matrices under this action. In other words, $\convperms^d_n$ is the subgroup of permutations of the columns of a $(d+1) \times n$ matrix with positive maximal minors such that the resulting matrix has again positive maximal minors.\par
It turns out that the parity of $d$ has a large impact on the structure of $\convperms^d_n$. We give a full description of this structure in \Cref{prop:convex autos odd} and \Cref{prop:convex autos even} below. In fact, the group $\convperms^d_n$ is a subgroup of the automorphisms of a cyclic $d$-polytope with $n$ vertices. To see this, we recall some elementary theory of (cyclic) polytopes.

\begin{lmm}\label{lmm:face condition}
    Let $P$ be a $d$-dimensional polytope with vertex set $V := \{x_1, \dots, x_n\}$. Then $F = \{x_{i_1}, \dots, x_{i_d}\}$ is the vertex set of a facet if and only if
    \begin{equation}\label{eq:face cond det}
    \det
    \begin{pmatrix}
       1 & \dots & 1 & 1\\
   x_{i_1} & \dots & x_{i_d} & y
  \end{pmatrix}
 \end{equation}
 has a fixed sign for all $y \in V-F$.
\end{lmm}
\begin{proof}
    Note that
    $$\det
    \begin{pmatrix}
    1 & \dots & 1 & 1\\
   x_{i_1} & \dots & x_{i_d} & y
  \end{pmatrix} = \det
    \begin{pmatrix}
   x_{i_2} - x_{i_1} & \dots & x_{i_d} - x_{i_1} & y - x_{i_1}
  \end{pmatrix}$$
  As a function in $y$, this determinant vanishes exactly on the hyperplane spanned by $x_{i_1},\dots,x_{i_d}$ and has constant sign on the two associated open half-spaces.
\end{proof}

\begin{corollary}[Gale's evenness criterion, \cite{Gale1963NeighborlyAC}]\label{cor:gale}
    Let $P$ be a polytope with vertices $x_1, \dots, x_n$ such that all maximal minors of \eqref{eq:matr}
have the same sign. Then the facets of $P$ are exactly the sets $F= x_I := \{x_{i_1}, \dots, x_{i_d}\}$ such that $\#\{i \in I | \ i > j\}$ has the same parity for all $j \in [n] - I$. In other words, $P$ is a cyclic polytope.
\end{corollary}
\begin{proof}
    This follows immediately from \Cref{lmm:face condition} as the sign of the determinant \eqref{eq:face cond det} is exactly $(-1)^{\#\{i \in I | \ i > j\}}$ times the sign of a maximal minor of \eqref{eq:matr} for $y=x_j\notin F$.
\end{proof}

\begin{corollary}\label{cor:conv pres is aut}
Let $x_1,\dots,x_n$ be such that the matrix
    \begin{equation*}
      \begin{pmatrix}
      1 & \dots & 1\\
       x_1 & \dots & x_n
      \end{pmatrix}
 \end{equation*}
 has positive maximal minors. Then for every $\pi \in S_n$ such that the matrix
 \begin{equation*}
      \begin{pmatrix}
      1 & \dots & 1\\
       x_{\pi(1)} & \dots & x_{\pi(n)}
      \end{pmatrix}
 \end{equation*}
 has positive maximal minors, $x_i \mapsto x_{\pi(i)}$ is a combinatorial automorphism of the cyclic polytope $P = \convhull(x_1, \dots, x_n)$. That is, it defines an automorphism of its face lattice.
\end{corollary}
\begin{proof}
    This is clear from \Cref{cor:gale} since the face condition there is only a condition on indices: it does not depend on the $x_i$ themselves.
\end{proof}

We will see in \Cref{cor:ev dim conv pres} that for even $d$ the converse is true up to sign, that is, combinatorial automorphisms preserve the property that all minors have the same sign. This is not true in odd dimensions:

\begin{example}\label{ex:counter convex}
 Let $x_1,\dots, x_6\in\mathbb R^3$ be such that
 \begin{equation*}
  \begin{pmatrix}
  1 & \dots & 1\\
   x_1 & \dots & x_6
  \end{pmatrix}
 \end{equation*}
 has positive maximal minors. Then $(x_1,x_2,x_3,x_4,x_5,x_6)\mapsto(x_6,x_2,x_3,x_4,x_5,x_1)$ is an automorphism of cyclic polytopes, but
 \begin{equation*}
 \det
  \begin{pmatrix}
  1 & 1 & 1 & 1\\
   x_6 & x_2 & x_3 & x_4
  \end{pmatrix} < 0
 \end{equation*}
while
 \begin{equation*}
 \det
  \begin{pmatrix}
  1 & 1 & 1 & 1\\
   x_2 & x_3 & x_4 & x_5
  \end{pmatrix}>0
 \end{equation*}
\end{example}
Let us now give a full description of $\convperms^d_n$. We will use the following characterization of the combinatorial automorphisms of a cyclic polytope:

\begin{theorem}[{\cite[Theorem 8.3]{KaibelAutomorphismGO}}]\label{thm:cyclaut}
The combinatorial automorphism group of a cyclic $d$-polytope with $n$ vertices is isomorphic to\par
\vspace{1em}
{
\setlength\tabcolsep{2em}
\renewcommand{\arraystretch}{2}
\centering
\begin{tabular}[t]{cccc}
         & $n = d + 1$ & $n = d + 2$ & $n \geq d+3$ \\ \hline
$d$ even & $\mb S_n$   & $\mb S_{\frac n 2} \text{ wr } \mb Z_2$                     & $\mb D_n$ \\ \hline
$d$ odd  & $\mb S_n$   & $\mb S_{\lceil{\frac n 2}\rceil} \times \mb S_{\lfloor{\frac n 2}\rfloor}$ & $\mb Z_2 \times \mb Z_2$
\end{tabular}\par}
\vspace{1em}
\end{theorem}

We start with the case of odd dimension.

\begin{proposition}\label{prop:convex autos odd}
Assume $d$ is odd. Then the group $\convperms_n^d$ is
\begin{enumerate}[i)]
    \item $A_n$ if $n=d+1$,
    \item $A_n\cap (S_{\frac{n-1} 2}\times S_{\frac{n+1} 2})$ if $n=d+2$,
    \item $\mathbb Z/2$ if $n \geq d+3$ and $\frac{d+1}{2}$ is even and
    \item  $1$ if $n \geq d+3$ and $\frac{d+1}{2}$ is odd.
\end{enumerate}
\end{proposition}
For the proof we need the following small lemma:
\begin{lmm}\label{lmm:sign switch}
    Let \begin{equation*}
      X = \begin{pmatrix}
       x_1 & \dots & x_{n+1}
      \end{pmatrix}
 \end{equation*}
 be a $n \times (n+1)$-matrix such that all minors have the same sign. Then switching two even or two odd columns switches the sign of all minors of $X$.
\end{lmm}
\begin{proof}
    Let $i,j$ be the indices of the columns that are switched and write $X'$ for the matrix obtained from $X$ by switching columns $i$ and $j$. For $I = \{i_1, \dots, i_n\}$ with $i_1 < \dots < i_n$ we write $X_I$ for the $n\times n$ matrix $\begin{pmatrix} x_{i_1} & \dots & x_{i_n} \end{pmatrix}$.\\
    If $I$ contains both $i$ and $j$, then $X'_I$ is obtained from $X_I$ by switching two columns and thus inverts the sign of the determinant. Now assume that $I$ does not contain $j$, so it contains $i$. $X'_I$ is obtained from $X_I$ by replacing $x_i$ with $x_j$. In particular, if $J := I \cup \{j\}- \{i\}$, then $X_I'$ has the same set of columns as $X_J$. Now note that every integer between $i$ and $j$ is contained in $I$ and since $i-j$ is even the number of such integers is odd. Thus, an odd number of transpositions is required to obtain $X_J$ from $X_I'$, and thus the determinant of $X_I'$ and the determinant of $X_J$ have inverse signs.
\end{proof}

\begin{proof}[Proof of \Cref{prop:convex autos odd}]
    Let $x_1,\dots,x_n$ be such that \eqref{eq:matr} is positive. Set $P = \convhull(x_1, \dots, x_n)$. We need to determine the subgroup of $\Aut(P)$ consisting of automorphisms that preserve positivity of all minors. Given $\pi \in \Aut(P)$ we write $\pi(X)$ for the matrix whose $i$-th column is the $\pi(i)$-th column of $X$.
    \begin{itemize}
        \item $i)$ is clear from \Cref{thm:cyclaut} as the determinant is an alternating map (here $P$ is a simplex).
    \item For $ii)$ we use that by \Cref{thm:cyclaut} we have $\Aut(P)= S_{\frac{n-1} 2}\times S_{\frac{n+1} 2}$ if $n=d+2$, where the first factor acts on the even and the second on the odd vertices of $P$. Thus, the statement follows from \Cref{lmm:sign switch}.
   \item For $iii)$ and $iv)$ we use that, again by \Cref{thm:cyclaut}, $\Aut(P)=\mb Z /2 \times \mb Z/2$ where the first factor acts by switching the first and the last vertex and the second factor acts by inverting the order on the inner vertices. Let $\pi$ be the generator of the first factor and $s$ the generator of the second factor. If $\frac{d+1}{2}$ is even, then $\tau = \pi \circ s$ preserves the sign of minors, otherwise it flips the sign. This is because any minor of $\pi \circ s(X)$ is turned into a submatrix of $X$ by $\frac{d+1}{2}$ transpositions. Next, aiming for a contradiction, assume $s$ preserves positive minors. Then $\pi \circ s \circ s = \pi$ either preserves or flips the sign of all minors. But this is not the case: Consider the matrix of the first $d+1$ columns of $\pi(X)$. This matrix is turned into a submatrix of $X$ by $d$ transpositions of columns, so it has negative determinant. But the matrix of columns $2,\dots,d+2$ of $\pi(X)$ is still a submatrix of $X$ and has, in particular, positive determinant (see\ \Cref{ex:counter convex} for an example in the case $d=3$).
    \end{itemize}
\end{proof}

\begin{proposition}\label{prop:convex autos even}
Assume $d$ is even. Then the group $\convperms_n^d$ is
\begin{enumerate}[i)]
    \item $A_n$ if $n=d+1$,
    \item $(A_n\cap (S_{ \frac n 2}\times S_{\frac n 2})) \rtimes \mb Z/2$ if $n=d+2$ and $\frac{d}{2}$ is even,
    \item $\ker \phi$ for the map $\phi: S_{ \frac n 2}\times S_{\frac n 2} \rtimes \mb Z/2 \to \{-1,1\}$ that maps $(\omega,\pi,\tau)$ to $\sgn(\omega)\sgn(\pi)\gamma(\tau)$ (where $\gamma$ maps $\tau$ to $-1$), if $n=d+2$ and $\frac{d}{2}$ is odd,
    \item $\mb D_n$ if $n \geq d+3$ and $\frac{d}{2}$ is even and
    \item  $\mb Z/n$ if $n \geq d+3$ and $\frac{d}{2}$ is odd.
\end{enumerate}
\end{proposition}
\begin{proof}
        Let again $x_1,\dots,x_n$ be such that \eqref{eq:matr} is positive. We set $P = \convhull(x_1, \dots, x_n)$.
    \begin{itemize}
        \item $i)$ is still clear from \Cref{thm:cyclaut} as the determinant is an alternating map (again, $P$ is a simplex).
    \item For $ii)$ we need to adapt the proof of \Cref{thm:cyclaut}. Let $s \in S_n$ denote the order-reversing permutation. Since $\frac{d}{2}$ transpositions of columns turn any $d+1\times d+1$ submatrix of $s(X)$ into a submatrix of $X$, $s$ preserves the sign of minors if $\frac{d}{2}$ is even and flips it if $\frac{d}{2}$ is odd. In the first case, going through the proof of \Cref{thm:cyclaut} and using \Cref{lmm:sign switch}, we obtain the semi-direct product $(A_n \cap S_{\frac{n} 2}\times S_{\frac{n} 2}) \rtimes \mb Z/2$. In the second case it is more difficult to describe the subgroup; we can define it as the kernel of the map $\phi: (S_{\frac{n} 2}\times S_{\frac{n} 2}) \rtimes \mb Z/2 \to \{-1,1\}$ that maps $(\omega,\pi,\tau)$ to $\sgn(\omega)\sgn(\pi)\gamma(\tau)$ where $\gamma$ maps $\tau$ to $-1$.
   \item For $iv)$ we use that, again by \Cref{thm:cyclaut}, $\Aut(P)=\mb D_n$. Let $r$ be rotation by $1$ and $s$ the order-reversing permutation. They generate $\mb D_n$ and since $\frac{d}{2}$ is even, both of them preserve positive minors.
   \item Similarly, if $\frac{d}{2}$ is odd, then $s$ will flip the signs of all minors. Thus, since $srs = r^{-1}$, we just obtain the subgroup generated by $r$ in this case.
    \end{itemize}
\end{proof}

For case v), compare \cite[Remark~3.3]{postnikov2006}.

\begin{corollary}\label{cor:ev dim conv pres}
    For even $d$, every automorphism of a cyclic $d$-polytope with $n$ vertices preserves the property that the maximal minors of \eqref{eq:matr} have constant sign.
\end{corollary}
\begin{proof}
    Going through the proof of \Cref{prop:convex autos even} again, we see that any combinatorial automorphism either preserves or flips the sign of all maximal minors.
\end{proof}

The following result adds further motivation to our interest in $C^d_n$ and explains why we call $\ourfeatures^d$ the ring of volume invariants.

\begin{proposition}\label{prop:eq rel on pwl}
For all $n \geq d+1$ there is a non-empty Zariski open subset $O$ of $\R^{d\times n}$
 such that for all $X\in O$, $\sigma \in S_n$:
 $$\langle S(X),\signedvolume_d\rangle=\langle S(\sigma.X), \signedvolume_d\rangle
 \textrm{ if and only if }
 \sigma \in \convperms^d_n.$$
\end{proposition}
Here, we identify a piecewise linear path $X$ with its ordered set of control points, i.e., a $d\times n$ matrix $X$.
\begin{proof}
The implication $\Leftarrow$ holds for all $X$ as it holds for the Zariski dense subset of $X$ with \eqref{eq:matr} positive, as discussed in the introduction. Thus it suffices to show the converse.\par
Let us denote the vanishing locus of the polynomial $$\langle S(X), \signedvolume_d \rangle-\langle S(\sigma.X), \signedvolume_d\rangle = H^d_n(\signedvolume_d) - \sigma.H^d_n(\signedvolume_d)$$ in $\mb R^{d \times n}$ by $Z_\sigma$ for $\sigma \in S_n$ and set $Z := \bigcup_{\sigma \in S_n \backslash \convperms^d_n} Z_\sigma$. This is the locus of $X$ violating the implication $\Rightarrow$. As it is closed, we only need to show that $Z$ is not the whole space. Then we can choose $O:=\mb R^{d \times n} - Z$.\par
Since $\mb R^{d \times n}$ is irreducible, it suffices to show that $Z_\sigma \not= \mb R^{d\times n}$ for all $\sigma \notin \convperms^d_n$. Thus, let $\sigma \in S_n \backslash \convperms^d_n$. Then we can choose $X$ such that \eqref{eq:matr} is positive for $X$, but not for $\sigma.X$. We will now argue by contraposition that the impliciation $\Rightarrow$ is true for $X$.\par
Let $x_1,\dots,x_n$ denote the columns of $X$. Set $P= \convhull(x_1,\dots,x_n)$. Then the matrix
\begin{equation*}
 \begin{pmatrix}
 1 & \dots & 1\\
 x_{\sigma^{-1}(1)} & \dots & x_{\sigma^{-1}(n)}
 \end{pmatrix}
\end{equation*}
has a negative maximal minor. There is some triangulation $\mc S$ of $P$ containing the simplex corresponding to the index set of this minor. To $\Delta \in \mc S$ we associate the indices of its vertices $\{ i^\Delta_1,\dots,i^\Delta_{d+1}\}, \ i_1 < \dots < i_{d+1}$. Then we consider
$$\signedvolume_d'(X):= \sum_{\Delta \in \mc S} \det\begin{pmatrix}1&\dots&1\\x_{i^\Delta_1}&\dots&x_{i^\Delta_{d+1}}\end{pmatrix},$$
Since $\Delta$ is a triangulation and since all determinants appearing in the sum are positive if \eqref{eq:matr} is positive, $\signedvolume_d'$ agrees with $\langle S(X), \signedvolume_d\rangle$ on the Zariski dense subset of $X$ with \eqref{eq:matr} positive and thus $\signedvolume_d'$ and $\langle S(X), \signedvolume_d\rangle$ are identical. But by construction $\signedvolume'_d(\sigma.X)$ is strictly bounded by the volume of $P$ (since at least one minor appearing in the sum will be negative) and so
\begin{equation*}
    \langle S(X), \signedvolume_d\rangle - \langle S(\sigma.X),\signedvolume_d \rangle > 0.\qedhere
\end{equation*}

\end{proof}

In particular, away from some exceptional closed set, we have the implication
$$\langle S(X), \signedvolume_d\rangle = \langle S(\sigma.X), \signedvolume_d\rangle \Rightarrow \forall \ w \in \ourfeatures^d: \ \langle S(X),w \rangle = \langle S(\sigma.X), w \rangle $$
for all $X$.
\section{Investigating the ring of volume invariants}\label{sec:ring_of_invariants}

In the following we write $\ourfeatures^d_n:=\ourfeatures^d_n(\convperms^d_n)$ for simplicity. Given the case distinction in Proposition \ref{prop:convex autos even}, we will treat the cases $n\geq d+3$ simultanuously and put 
$$\ourfeatures_{\geq d+3}^d:=\bigcap_{n\geq d+3} \ourfeatures_n^d$$
so that we have
$$\ourfeatures^d=\ourfeatures_{d+1}^d\cap \ourfeatures_{d+2}^d \cap \ourfeatures_{\geq d+3}^d.$$

Note that the kernel $\mc I(\pwlinear^d_{d+2})$ of $\homomor^d_{d+2}$ is contained in $\ourfeatures^d_{d+1} \cap \ourfeatures^d_{d+2}$.
We can write
\begin{equation*}
 \ourfeatures^d=\frac{\ourfeatures^d}{\ourfeatures^d \cap \mathcal{I}(\pwlinear^d_{d+2})}\oplus\big(\ourfeatures^d_{\geq d+3}\cap\mathcal{I}(\pwlinear^d_{d+2})\big)
\end{equation*}

While we can give a description of $\ourfeatures^d_{\geq d+3}$, the problem for $\ourfeatures^d_{d+1}$ and $\ourfeatures^d_{d+2}$ is more difficult. We give examples and a conjecture instead, based on computations for low values of $d$:
\begin{conjecture}\label{conj:only vol for d+2}
$$\ourfeatures^d_{d+2}/\mathcal{I}(\pwlinear^d_{d+2}) \cong \mb R[H^d_{d+2}(\signedvolume_d)]\subseteq \mc S^d[x_1,\dots,x_{d+2}]$$
i.e.\ the subalgebra of $\mc S^d[x_1,\dots,x_{d+2}]$ generated by $H^d_{d+2}(\signedvolume_d)$. Equivalently,
$$\ourfeatures^d_{d+2}=\spann\{(\signedvolume_d)^{\shuffle k},k\geq 0\}\oplus \mathcal{I}(\pwlinear_{d+2}^d).$$
\end{conjecture}

\begin{example}
    A computation using \textsc{Macaulay 2} shows that the vector space of invariants of degree $\leq 6$ in $\ourfeatures^3_{4}$ is spanned by $18$ elements, one in degree $3$ (the signed volume), $6$ in degree $5$ and $11$ in degree $6$ (including the shuffle square of the signed volume). Here are two examples:
    \begin{itemize}
    \item $w_1:=\texttt{12333} + \texttt{13233} -\frac{2}{3} \cdot \texttt{13323} - \frac{4}{3} \cdot \texttt{13332} - \texttt{21333} - \texttt{23133} + \frac{2}{3} \cdot \texttt{23313} + \frac{4}{3} \cdot \texttt{23331} + \frac{5}{3} \cdot \texttt{31323} - \frac{2}{3} \cdot \texttt{31332} - \frac{5}{3} \cdot \texttt{32313} + \frac{2}{3} \cdot \texttt{32331} + \texttt{33132} - \texttt{33231} + \texttt{33312} - \texttt{33321}$
    \item $w_2 :=
         - \texttt{123333} -3 \cdot \texttt{132333} + 4 \cdot \texttt{133233} + \texttt{213333} + 3 \cdot \texttt{231333} +\\
        -4 \cdot \texttt{233133} + \texttt{312333} -2 \cdot \texttt{313233}
        - \texttt{321333} +\\
        2 \cdot \texttt{323133} + 2 \cdot \texttt{331323}  -4 \cdot \texttt{331332}  -2 \cdot \texttt{332313} +\\
        4 \cdot \texttt{332331}  -\texttt{333123} + 3 \cdot \texttt{333132} + \texttt{333213} +\\
        -3 \cdot \texttt{333231} + \texttt{333312} - \texttt{333321}$
    \end{itemize}
    Writing $x_i=(x_{i1}, \dots, x_{id})$ and $a_i:= x_{i+1} - x_i$, their images under \eqref{eq:sign map} in the polynomial ring $\mb R[x_1,\dots,x_d]$ are given by
    \begin{small}
    \begin{align*}
             &2\cdot(-3a_{13}^3a_{22}a_{31} + 3a_{12}a_{13}^2a_{23}a_{31} - 4a_{13}^2a_{22}a_{23}a_{31} + 4a_{12}a_{13}a_{23}^2a_{31} \\ 
 &- 4a_{13}a_{22}a_{23}^2a_{31} + 4a_{12}a_{23}^3a_{31} + 3a_{13}^3a_{21}a_{32} - 3a_{11}a_{13}^2a_{23}a_{32} \\ 
 &+ 4a_{13}^2a_{21}a_{23}a_{32} - 4a_{11}a_{13}a_{23}^2a_{32} + 4a_{13}a_{21}a_{23}^2a_{32} - 4a_{11}a_{23}^3a_{32} \\ 
 &- 3a_{12}a_{13}^2a_{21}a_{33} + 3a_{11}a_{13}^2a_{22}a_{33} - 4a_{12}a_{13}a_{21}a_{23}a_{33} + 4a_{11}a_{13}a_{22}a_{23}a_{33} \\ 
 &- 4a_{12}a_{21}a_{23}^2a_{33} + 4a_{11}a_{22}a_{23}^2a_{33} - 2a_{13}^2a_{22}a_{31}a_{33} + 2a_{12}a_{13}a_{23}a_{31}a_{33} \\ 
 &- 4a_{13}a_{22}a_{23}a_{31}a_{33} + 4a_{12}a_{23}^2a_{31}a_{33} + 2a_{13}^2a_{21}a_{32}a_{33} - 2a_{11}a_{13}a_{23}a_{32}a_{33} \\ 
 &+ 4a_{13}a_{21}a_{23}a_{32}a_{33} - 4a_{11}a_{23}^2a_{32}a_{33} - 2a_{12}a_{13}a_{21}a_{33}^2 + 2a_{11}a_{13}a_{22}a_{33}^2 \\ 
 &- 4a_{12}a_{21}a_{23}a_{33}^2 + 4a_{11}a_{22}a_{23}a_{33}^2 - 3a_{13}a_{22}a_{31}a_{33}^2 + 3a_{12}a_{23}a_{31}a_{33}^2 \\ 
 &+ 3a_{13}a_{21}a_{32}a_{33}^2 - 3a_{11}a_{23}a_{32}a_{33}^2 - 3a_{12}a_{21}a_{33}^3 + 3a_{11}a_{22}a_{33}^3)
        \end{align*}
        \end{small}
        and 
        \begin{small}
    \begin{align*}
            &24\cdot (-a_{13}^4a_{22}a_{31} + a_{12}a_{13}^3a_{23}a_{31} - 2a_{13}^3a_{22}a_{23}a_{31} + 2a_{12}a_{13}^2a_{23}^2a_{31} \\ 
 &+ a_{13}^4a_{21}a_{32} - a_{11}a_{13}^3a_{23}a_{32} + 2a_{13}^3a_{21}a_{23}a_{32} - 2a_{11}a_{13}^2a_{23}^2a_{32} \\ 
 &- a_{12}a_{13}^3a_{21}a_{33} + a_{11}a_{13}^3a_{22}a_{33} - 2a_{12}a_{13}^2a_{21}a_{23}a_{33} + 2a_{11}a_{13}^2a_{22}a_{23}a_{33} \\ 
 &- a_{13}^3a_{22}a_{31}a_{33} + a_{12}a_{13}^2a_{23}a_{31}a_{33} + a_{13}^3a_{21}a_{32}a_{33} - a_{11}a_{13}^2a_{23}a_{32}a_{33} \\ 
 &- a_{12}a_{13}^2a_{21}a_{33}^2 + a_{11}a_{13}^2a_{22}a_{33}^2 + a_{13}^2a_{22}a_{31}a_{33}^2 - a_{12}a_{13}a_{23}a_{31}a_{33}^2 \\ 
 &+ 2a_{13}a_{22}a_{23}a_{31}a_{33}^2 - 2a_{12}a_{23}^2a_{31}a_{33}^2 - a_{13}^2a_{21}a_{32}a_{33}^2 + a_{11}a_{13}a_{23}a_{32}a_{33}^2 \\ 
 &- 2a_{13}a_{21}a_{23}a_{32}a_{33}^2 + 2a_{11}a_{23}^2a_{32}a_{33}^2 + a_{12}a_{13}a_{21}a_{33}^3 - a_{11}a_{13}a_{22}a_{33}^3 \\ 
 &+ 2a_{12}a_{21}a_{23}a_{33}^3 - 2a_{11}a_{22}a_{23}a_{33}^3 + a_{13}a_{22}a_{31}a_{33}^3 - a_{12}a_{23}a_{31}a_{33}^3 \\ 
 &- a_{13}a_{21}a_{32}a_{33}^3 + a_{11}a_{23}a_{32}a_{33}^3 + a_{12}a_{21}a_{33}^4 - a_{11}a_{22}a_{33}^4)
        \end{align*}
        \end{small}
    respectively.
\end{example}

Let us now give a description of $\ourfeatures^d_{\geq d+3}$. Let $\mc A$ denote the antipode of the Hopf algebra $\mb R\langle \texttt 1, \dots, \texttt d\rangle$, that is, the map sending a word $w$ to $(-1)^{d+1}w'$ where $w'$ is obtained from $w$ by reversing the order of its letters.

\begin{definition}
We define $\timerevinv^d:=\{w\in \mb R\langle \texttt{1},\dots, \texttt{d}\rangle \ | \ w=\antipode w\}$. This is the subring of $w \in \mb R\langle \texttt{1},\dots, \texttt{d}\rangle$ such that $\langle S(X), w\rangle = \langle S(X^{-1}), w \rangle$ for all paths $X:[0,1] \to \mb R^d$, where $X^{-1}: [0,1]\to \mb R^d,\  t \mapsto X(1-t)$.
Indeed, taking the antipode is the adjoint operation to time reversal (see e.g.\ \cite{preiss2024algebraic}).
\end{definition}

In the following, we will make crucial use of the following theorem (immediately equivalent to \cite[Lemma~5.2]{areasofareas}), which can be seen as a corollary of the Chen-Chow theorem for piecewise linear paths, or a weaker version thereof:

\begin{theorem}[{weak Chen-Chow}]\label{thm:chenchow}
    Let $w, v \in \mb R\langle \mathtt 1, \dots, \mathtt d \rangle$. If $\langle S(X), w\rangle = \langle S(X), v \rangle$ for all piecewise linear paths $X$, then $w=v$.
\end{theorem}
\begin{proof}
By the proof of \cite[Lemma~8]{diehl2019invariants}, the image of piecewise linear paths under $S$ spans the dual $(\mb R\langle \texttt 1, \dots, d \rangle_{\leq k})^*$ of the vector space of length $\leq k$ words.
\end{proof}

\begin{proposition}\label{prop:inv geq3 odd}
Assume $d$ is odd. Then 
$$\ourfeatures_{\geq d+3}^d=\timerevinv^d$$ if $\frac{d+1}2$ is even, and $\ourfeatures_{\geq d+3}^d=\mb R\langle \texttt{1},\dots, \texttt{d}\rangle$ if $\frac{d+1}{2}$ is odd.
\end{proposition}
\begin{proof}
    Let $X$ be a piecewise linear path in $\mb R^d$ with $n\geq d+3$ control points. If $\frac{d+1}{2}$ is odd then $\convperms^d_n$ is trivial by \Cref{prop:convex autos odd}. If $\frac{d+1}{2}$ is even, then the proof of \Cref{prop:convex autos odd} shows that $\convperms^d_n$ is generated by the reflection $\tau$ which inverts the order of the vertices. On piecewise linear paths, this corresponds to $X \mapsto X^{-1}$. Thus, if $w \in \ourfeatures_{\geq d+3}$, then $\langle S(X), \mc A w \rangle = \langle S(X^{-1}), w\rangle = \langle S(X), w\rangle$, concluding the proof by \Cref{thm:chenchow}.
\end{proof}

Note that $\timerevinv$ can be identified as the image of $$\mb R\langle \texttt{1},\dots, \texttt{d}\rangle \to \mb R\langle \texttt{1},\dots, \texttt{d}\rangle, \ w \mapsto w + \mc Aw.$$

\begin{example}
In $d=3$, consider the concatenation square of signed 3-volume
\begin{equation*}
\signedvolume_3^{\bullet 2}=(\texttt{123}+\texttt{231}+\texttt{312}-\texttt{213}-\texttt{132}-\texttt{321})^{\bullet 2},
\end{equation*}
where the concatenation (of words) $\bullet$ is the bilinear non-commutative product on $\mb R\langle\mathtt{1},\dots,\mathtt{d}\rangle$ given by
\begin{equation*}
 \mathtt{i}_1\dots\mathtt{i}_k\bullet\mathtt{i}_{k+1}\dots\mathtt{i}_m=\mathtt{i}_1\dots\mathtt{i}_m
\end{equation*}
We have $\antipode\signedvolume_3^{\bullet 2}=\signedvolume_3^{\bullet 2}$, so 
$\signedvolume_3^{\bullet 2}\in\timerevinv^3=\ourfeatures_{\geq 6}^3$.

Furthermore, $\langle S(X),\signedvolume_3^{\bullet 2}\rangle$ is the integral
\begin{align*}
\int_{0\leq t_1\leq\dots\leq t_6\leq 1}\det(X'(t_1),X'(t_2),X'(t_3))\det(X'(t_4),X'(t_5),X'(t_6))dt_1\dots dt_6&
\end{align*}
so if $X$ has only 4 segments,
there is no choice of $0\leq t_1\leq\dots\leq t_6\leq 1$ such that both determinants are non-zero.
Thus, $\signedvolume_3^{\bullet 2}\in \ourfeatures_{\geq 6}^3\cap\mathcal{I}(\pwlinear_5^3)\subset \ourfeatures^3$.
\end{example}

\begin{definition}
We define $\loopclosureinv^d$ as the subring of $u \in \mb R\langle\mathtt{1},\dots,\mathtt{d}\rangle$ such that
\begin{equation*}
 \langle S(X),u\rangle=\langle S(X\sqcup \{X(1)\to X(0)\}),u\rangle=\langle S(\{X(1)\to X(0)\}\sqcup X),u\rangle,
\end{equation*}
where $\sqcup$ is concatenation of paths
and $\{X(1)\to X(0)\}$ is the linear segment from $X(1)$ to $X(0)$. That is, the elements of $\loopclosureinv^d$ correspond to signature values that are stable
both under closing a path to a loop by concatenating a linear segment to the right (the \textit{right loop closure}),
as well as under closing a path to a loop by concatenating a linear segment to the left (the \textit{left loop closure}). We refer to \cite{loopcl24} for details.
\end{definition}

For example, the signatures of the five paths from \Cref{fig:2polytope} agree on all loop closure invariants
(since the right closure of the first path is the left closure of the second path depicted, and so on).

\begin{proposition}\label{prop:inv_loopclosure_timerev}
Assume $d$ is even. Then 
$$\ourfeatures_{\geq d+3}^d=\loopclosureinv^d\cap\timerevinv^d$$ if $\frac{d}{2}$ is even, and $\ourfeatures_{\geq d+3}^d=\loopclosureinv^d$ if $\frac{d}{2}$ is odd.
\end{proposition}
\begin{proof}
    If $\frac{d}{2}$ is even then $\convperms^d_n$ is the group $\mb D_n$ by \Cref{prop:convex autos even}. It is spanned by the order-reversing permutation $s_n$ and the rotation $r_n$. As observed in (the proof of) \Cref{prop:inv geq3 odd}, invariants under $s_n$ for all $n$ simultaneously are precisely the elements of $\timerevinv$. On the other hand, invariants under all $r_n$ for all $n$ simultaneously are the elements of $\loopclosureinv$, see \cite{loopcl24}. 
    Indeed, for $w \in \ourfeatures^d_{\geq d+3}$ and the right closure $\bar X^R$ of the path $X$ (which is a piecewise linear path with $n+1$ control points) we must have $\langle S(r_{n+1}(\bar X^R)), w \rangle = \langle S( \bar X), w \rangle$ for the rotation $r_{n+1} \in \mb Z/(n+1)$. But $\langle S(r_{n+1}(\bar X^R)), w \rangle  = \langle S(X), w \rangle $ by reparametrisation invariance of iterated integrals. Similarly, for the left closure $\bar X^L$ of the path $X$ we must have $\langle S( r^{-1}_{n+1}(\bar X^L)), w \rangle = \langle S(\bar X), w \rangle$. But again, $\langle S(r^{-1}_{n+1}(\bar X^L)), w \rangle  = \langle S(X), w \rangle$. Using \Cref{thm:chenchow} and that both left- and right-closure admit an adjoint \cite[Lemma 4.7]{loopcl24} we see that $w \in \loopclosureinv$. The reverse inclusion is immediate from \cite[Proposition 4.3]{loopcl24}. \par
    In the case that $\frac{d}{2}$ is odd we have that $\convperms^d_n = \mb Z/n$ is just generated by $r_n$ as shown in \Cref{prop:convex autos even}. Thus, both statements follow.
\end{proof}

\begin{theorem}\label{thm:abundant invariants}
 For any $d$,
  $\ourfeatures_{\geq d+3}^d\cap\mathcal{I}(\pwlinear^d_{d+2})$ (and in particular $\ourfeatures^d$) contains infinitely many algebraically independent elements (with respect to the shuffle product) and is thus in particular infinitely generated as a (shuffle) subalgebra of $\R\langle\word{1},\dots,\word{d}\rangle$.
\end{theorem}
\begin{proof}  
    If $d$ is odd then we have $\ourfeatures^d_{\geq d+3} = \timerevinv^d$ or $\ourfeatures^d = \mb R\langle \texttt 1, \dots, d\rangle$. If $d$ is even then $\ourfeatures^d_{\geq d+3} = \loopclosureinv^d \cap \timerevinv^d$ or $\ourfeatures^d_{\geq d+3} = \loopclosureinv^d$. We claim that all of these algebras contain infinitely many algebraically independent elements. This is true for $\mb R\langle \texttt 1, \dots, d\rangle$ as it is freely generated by the Lyndon words, see \cite[Theorem 6.1]{reutenauer1993free}. Then it is also true for $\timerevinv^d$: It is the kernel of the map (of vector spaces) $\psi: \mb R\langle \texttt 1, \dots, \texttt d\rangle \to \mb R\langle \texttt 1, \dots, \texttt d\rangle ,w\mapsto w - \mc A w$ and if it does not contain infinitely many algebraically independent elements, then there is a finite set $S$ of Lyndon words such that each element is already algebraic over $\mb R[S]$, thus contained in it. It follows that the image of $\psi$ would have to contain infinitely many algebraically independent $a_1,a_2,\dots$ but then the elements $a_1^{\shuffle 2}, a_2^{\shuffle 2}, \dots \in \timerevinv$ are still algebraically independent, yielding a contradiction.
    
In \cite{loopcl24} it is shown that $\loopclosureinv^d$ contains an infinite algebraically independent subset and by a similar argument as above
    it follows that the intersection $\loopclosureinv^d \cap \timerevinv^d$ also does, using that $w-\mc Aw$ is an element of $\loopclosureinv^d$ for every $w \in \loopclosureinv^d$ since $\mc Aw$ is a loop closure invariant if $w$ is (which follows from the definition).

    So we have shown that $\ourfeatures^d_{\geq d+3}$ contains infinitely many algebraically independent elements for any $d$. In particular, we can consider a composition
   \[ \begin{tikzcd}[nodes={inner sep=1pt}]
        \mb R[s_1,s_2,\dots] \rar[hookrightarrow] & \ourfeatures^d_{\geq d+3} \rar & \R\langle\texttt 1, \dots,\texttt d\rangle \rar & \R\langle\texttt 1, \dots,\texttt d\rangle / \mc I(\pwlinear^d_{d+2})
    \end{tikzcd}\]
    where the first map is injective. The kernel of this composition must contain infinitely many algebraically independent elements as the quotient on the right is isomorphic to a subring of $\mb R[x_1,\dots,x_{d+2}]$ via \eqref{eq:sign map}. Indeed, otherwise there is again some $N$ such that any element of the kernel is already algebraic over $\mb R[s_1,\dots,s_N]$ and we get an injection $\mb R[s_{N+1},s_{N+2},\dots] \to \mb R[x_1,\dots,x_{d+2}]$ which is absurd. The infinitely many algebraically independent elements are still algebraically independent in the larger ring $\ourfeatures^d_{\geq d+3}$ and by construction contained in $\mc I(\pwlinear^d_{d+2})$, proving the claim.
    
\end{proof}
\section{Outlook}

\paragraph{Computing volume invariants for even $d$}
In dimension $2$, we simply have $\ourfeatures^2=\loopclosureinv^2$.
A lowest-degree example of a loop closure invariant for two-dimensional paths that is independent of signed area is the following:
\begin{small}
\begin{align*}
 & \hphantom{\hspace{1.1em}} 3\cdot\texttt{111222}+5\cdot\texttt{112122}+3\cdot\texttt{112212}-3\cdot\texttt{112221}+3\cdot\texttt{121122}+\texttt{121212}\\
&-5\cdot\texttt{121221}+\texttt{122112}-5\cdot\texttt{122121}-3\cdot\texttt{122211}-3\cdot\texttt{211122}-5\cdot\texttt{211212}\\
&+\texttt{211221}-5\cdot\texttt{212112}
+\texttt{212121}+3\cdot\texttt{212211}-3\cdot\texttt{221112}+3\cdot\texttt{221121}\\
&+5\cdot\texttt{221211}+3\cdot\texttt{222111}
\end{align*}
\end{small}

However, starting from four dimensions the even case gets vastly more involved.
The lowest-degree generators of $\mathcal{I}(\pwlinear_6^4)$ are the following $8$ on level~$7$,
\begin{align*}
  &\signedvolume_4\bullet\signedvolume_3(\texttt{1},\texttt{2},\texttt{3}),\quad\!\!\!\signedvolume_4\bullet\signedvolume_3(\texttt{1},\texttt{2},\texttt{4}),\quad\!\!\!\signedvolume_4\bullet\signedvolume_3(\texttt{1},\texttt{3},\texttt{4}),\quad\!\!\!\signedvolume_4\bullet\signedvolume_3(\texttt{2},\texttt{3},\texttt{4}),\\
  &\signedvolume_3(\texttt{1},\texttt{2},\texttt{3})\bullet\signedvolume_4,\quad\!\!\!\signedvolume_3(\texttt{1},\texttt{2},\texttt{4})\bullet\signedvolume_4,\quad\!\!\!\signedvolume_3(\texttt{1},\texttt{3},\texttt{4})\bullet\signedvolume_4,\quad\!\!\! \signedvolume_3(\texttt{2},\texttt{3},\texttt{4})\bullet\signedvolume_4
\end{align*}
  where
  \begin{align*}
\signedvolume_4 &:=\signedvolume_4(\texttt{1},\texttt{2},\texttt{3},\texttt{4})\\
 &:=\texttt{1234}-\texttt{1243}-\texttt{1324}+\texttt{1342}+\texttt{1423}-\texttt{1432}-\texttt{2134}+\texttt{2143}\\
 &\hphantom{:=}+\texttt{2314}-\texttt{2341}-\texttt{2413}+\texttt{2431}+\texttt{3124}-\texttt{3142}-\texttt{3214}+\texttt{3241}\\
 &\hphantom{:=}+\texttt{3412}-\texttt{3421}-\texttt{4123}+\texttt{4132}+\texttt{4213}-\texttt{4231}-\texttt{4312}+\texttt{4321}
 \end{align*}
 
 is four dimensional signed volume and $$\signedvolume_3(\texttt{i},\texttt{j},\texttt{k})=\texttt{ijk}+\texttt{jki}+\texttt{kij}-\texttt{jik}-\texttt{ikj}-\texttt{kji}.$$

 \noindent No linear combination of these eight is a loop-closure invariant.\par
 Even though we know that $\ourfeatures^4_{\geq d+3}\cap\mathcal{I}(\pwlinear^4_6)$ is an infinitely generated subring,
 we need to compute very far to find the first generator.

\paragraph{The induced equivalence relation}

Recall that we can view $\ourfeatures^d$ as features on cyclic polytopes with $n\geq d+1$ vertices. If we consider the equivalence relation
$$(x_1, \dots, x_n) \sim (y_1, \dots, y_m) :\Leftrightarrow \forall \ f \in \ourfeatures^d: f(x_1, \dots, x_n) = f(y_1, \dots, y_m)$$
then we have $(x_1,...,x_n) \sim (x_{\sigma(1)},\dots,x_{\sigma(n)})$ for any $\sigma \in C^n_d$. The properties of iterated integrals imply that $(x_1,\dots,x_n) \sim (x_1 + c, \dots, x_n + c)$ for any $c \in \mb R^d$ and that $(x_1,\dots,x_n) \sim (x_1,\dots,\hat{x_i},\dots,x_n)$ ($x_i$ is omitted in the second tuple) whenever $x_i$ is a convex combination of $x_{i-1}$ and $x_{i+1}$.\par
However, we do not expect these three types of relations to generate $\sim$: For example, if \Cref{conj:only vol for d+2} holds true, then any two $d$-polytopes with $d+2$ vertices and the same volume are equivalent under $\sim$.
\begin{myproblem}
   How can the equivalence relation $\sim$ be described geometrically or combinatorial? Cf.\ \cite[Section~5]{loopcl24}, \cite[Conjecture~7.2]{diehl2019invariants} and \cite{diehllyonsnipreiss}.
\end{myproblem}

\paragraph{Specializing to $O$ and $SL$-invariants}

For $O_d$, the orthogonal group,
we may reduce to invariants on demand (cf.\ \cite{diehllyonsnipreiss}),
for example through a projection 
$$R:\,p\mapsto \int_{O_d}p(A x_1,\dots,A x_n) d\mu(A),$$
where $\mu$ is the Haar measure of $O_d$ with $\mu(O_d)=1$. 
$R$ preserves signature polynomials,
and is an example of a so-called Reynold's operator.
Now for $SL_d$, there is no finite Haar measure as it is a non-compact group.
However, we may still compute the intersection of the subrings of $\ourfeatures^d$ and the $SL_d$ invariants.
This yields functions on the positive Grassmannian,
as the latter can be represented by positive matrices modulo $SL_d$ action from the left. 
See for example \cite{derksenkemper} for the notions of Reynold's operator and Haar measure.

\paragraph{Invariants for other groups}

Instead of considering $G_n=C_n^d$, there are of course other interesting possibilities.\par
If we were to consider the maximal choice $G_n=S_n$, then $\ourfeatures^d(G)$ would be the zero subring. Indeed, take any piecewise linear path $P$, with $n$ vertices. Then one can double each vertex except the last to obtain a path with $2n-1$ vertices but the same signature. Permuting the order of the vertices allows then to obtain a tree-like path (i.e.\ a path with vanishing signature). Thus, any (simultaneous) invariant for the $S_n$-action must evaluate to $0$ under the signature. By (weak) Chen-Chow, Theorem~\ref{thm:chenchow}, this implies that the invariant itself must be $0$.

For $G_n=\mathbb{Z}/n$, we exactly get $\ourfeatures^d(G)=\loopclosureinv^d$,
and for $G_n=\mathbb{D}_n$, we have $\ourfeatures^d(G)=\loopclosureinv^d\cap\timerevinv^d$.

\paragraph{Acknowledgements}
The authors thank Carlos Améndola, Joscha Diehl, Jeremy Reizenstein, Leonard Schmitz, Bernd Sturmfels and Nikolas Tapia for helpful discussions and suggestions, and Shelby Cox and Gabriele Dian for reviewing a preliminary version of this paper. The authors furthermore thank Bernd Sturmfels for his support in restructuring the article.
The authors acknowledge support from DFG CRC/TRR 388 ``Rough Analysis, Stochastic Dynamics and Related Fields'', Project A04.
R.P.\ was affiliated with MPI MiS Leipzig until June 2024 and thanks MPI MiS for a guest status until December 2024 that facilitated the collaboration with F.L.
\begin{small}
\def\cprime{$'$}

\end{small}

\end{document}